\documentclass[a4paper,UKenglish,cleveref, autoref, thm-restate]{lipics-v2021}

\pdfoutput=1 
\hideLIPIcs 

\usepackage{graphicx}
\usepackage{amssymb,amsthm,amsmath,amsfonts,amstext,mathrsfs}
\usepackage{latexsym}
\usepackage{url,scalerel}
\usepackage{pdfpages}
\usepackage{tikz}
\usepackage{wrapfig}

\hypersetup{colorlinks=true,
linkcolor=blue,
citecolor=blue,
urlcolor=blue,
filecolor=blue,
bookmarksnumbered=true,
pdfstartview=FitH,
pdfhighlight=/N}

\newtheorem*{thmm}{Theorem}

\newtheorem{lem}{Lemma}

\makeatletter
\def\dual#1{\expandafter\dual@aux#1\@nil}
\def\dual@aux#1/#2\@nil{\begin{tabular}{@{}l@{}}#1\\#2\end{tabular}}
\makeatother

\input xy
\xyoption{all}

\title{Maximal subsets of $n\times n$ board without cycles}

\author{Giedrius Alkauskas}{Vilnius University, Department of mathematics and Informatics,\\
	Institute of Computer Science, Naugarduko 24, LT-03225 Vilnius, Lithuania} {giedrius.alkauskas@mif.vu.lt}{https://orcid.org/0000-0002-5747-363X}{}
\authorrunning{G. Alkauskas}
\titlerunning{Acyclic sets}

\Copyright{Giedrius Alkauskas}



\nolinenumbers
\hideLIPIcs{}

\usepackage{numprint}

\usepackage{mathtools}
\newcommand{\DeclareAutoPairedDelimiter}[3]{%
	\expandafter\DeclarePairedDelimiter\csname Auto\string#1\endcsname{#2}{#3}%
	\begingroup\edef\x{\endgroup
		\noexpand\DeclareRobustCommand{\noexpand#1}{%
			\expandafter\noexpand\csname Auto\string#1\endcsname*}}%
	\x}

\begin{document}
	\maketitle
\begin{abstract}
Let \(n\in\mathbb N\). Define \(\nabla(n)\) the least possible number of cells that must be removed from the \(n\times n\) square board in order to leave no cycles. Put $L_n=\left\lceil\left((n-1)^2+1\right)/3\right\rceil$. Luccio proved the general lower bound \(\nabla(n)\geq L_n\). Constructions of Luccio, Madelaine--Stewart, and Lien--Fu--Shih imply that $\nabla(n)=L_n$ for $n\equiv 1,4\pmod 6$, and in the remaining classes it is reduced to deciding whether \(\nabla(n)=L_n\) or \(\nabla(n)=L_n+1\). In this paper, we provide optimal constructions in the odd residue classes \(n\equiv3,5\pmod6\), (proving  $\nabla(n)=L_n$) and a geometric obstruction showing that equality in lower bound is impossible when \(n\equiv0,2\pmod6\) (yielding $\nabla(n)=L_n+1$). Starting with \(n=3\), the resulting sequence is given by
	\[
	\nabla(n)=\bigg{\lceil}\frac{n^2}{3}-\frac{n}{6}+\frac{4}{3}\bigg{\rceil}-\bigg{\lceil}\frac{n}{2}\bigg{\rceil}.
	\]
which is now listed in the OEIS as A354673.
\end{abstract}

\section{Previous results}
\footnote{Date: \today. First version of the manuscript: June 02, 2022}
The problem considered here is the square-grid instance of the feedback
vertex set problem \cite{weisstein}. Alternatively, this may be call {decycling number}. It is the smallest number of vertices needed to remove from graph in order to leave no cycles. For the \(m\times n\) grid graph, denote this number by $\nabla(m,n)$, and let \(\nabla(n)=\nabla(n,n)\).

The current paper aims at two targets. First, to extract all results for $\nabla(m,n)$  already existing in the literature and to specialize them to $m=n$. Second, in cases where the gap still remains, determine the exact value.

For $m\times n$ board, Luccio proved the general lower bound
\[
\nabla(m,n)\geq
\left\lceil\frac{(m-1)(n-1)+1}{3}\right\rceil.
\]
Thus, $\nabla(n)\geq
\left\lceil((n-1)^2+1)/3\right\rceil:=L_n$. We will re-derive the latter bound, for it will become crucial to express the excess $3L_n-(n-1)^2$  in terms of intrinsic graph parameters.\\

For the upper bound, one usually provides an explicit construction. If $n=2^r+1$, $r\geq 1$, Luccio gave an optimal recursive procedure matching the lower bound \cite{Luccio}. This construction can be easily extrapolated from a special example $n=2^3+1$ (see Figure \ref{ex2}, $n=9$). Madelaine and Stewart subsequently gave explicit examples for rectangular grids. In many congruence classes their constructions attain Luccio's lower bound, and in all remaining covered
cases they are within two of it \cite{MadelaineStewart}. In particular,
when specialized to square grids, their constructions settle the classes
\(n\equiv1,4\pmod6\), where the lower bound gives the correct answer \(\nabla(n)=L_n\). 

 In the other remaining residue classes for square gris their constructions leave a gap: for \(n\equiv0,2,3\pmod6\)
they give \(\nabla(n)\leq L_n+1\), while for \(n\equiv5\pmod6\) their direct
construction gives only \(\nabla(n)\leq L_n+2\).  Later, Lien, Fu and Shih
improved the general upper bound for grids, showing that for all
 \(m,n\geq 6\) one has
\[
\nabla(m,n)\leq
\left\lceil\frac{(m-1)(n-1)+1}{3}\right\rceil+1.
\]
Consequently, for square grids the remaining uncertainty is at most one: $L_n\leq \nabla(n)\leq L_n+1$, apart from small values, which can be checked directly
\cite{LienFuShih}.\\

Getting ahead of ourselves with a final answer in mind, the square case may therefore be summarized as follows:
\begin{eqnarray*}
\begin{array}{c|c|c}
	n\pmod6 & \text{previous status} & \text{new ingredient needed} \\ \hline
	0 & L_n\leq \nabla(n)\leq L_n+1
	& \text{prove }\nabla(n)\ne L_n \\[1mm]
	1 & \nabla(n)=L_n
	& \text{none} \\[2mm]
	2 & L_n\leq \nabla(n)\leq L_n+1
	& \text{prove }\nabla(n)\ne L_n \\[1mm]
	3 & L_n\leq \nabla(n)\leq L_n+1
	& \text{construct a set of size }L_n \\[1mm]
	4 & \nabla(n)=L_n
	& \text{none} \\[1mm]
	5 & L_n\leq \nabla(n)\leq L_n+1
	& \text{construct a set of size }L_n.
\end{array}
\end{eqnarray*}
(Classes $n\equiv 3,5\pmod 6$ both contain subclasses $n=2^{2r+1}+1$ and $n=2^{2r}+1$ where the sharp result was known.)
In the current paper, we add exactly the missing ingredients.  For
\(n\equiv3,5\pmod6\) the construction attaining $L_n$ is presented.  For \(n\equiv0,2\pmod6\), \(n\geq3\), we prove a geometric
obstruction showing that \(\nabla(n)=L_n\) is impossible.  Together with small initial boards, this completes the problem for square gird.
\begin{thmm}For $n\geq 3$, one has 
	\begin{eqnarray*}
		\nabla(n)=\bigg{\lceil}\frac{n^2}{3}-\frac{n}{6}+\frac{4}{3}\bigg{\rceil}-\bigg{\lceil}\frac{n}{2}\bigg{\rceil}.
	\end{eqnarray*}
\end{thmm}
\section{Background}
For $n\geq 2$ being a positive integer, consider the $n\times n$ board divided into $n^2$  unit cells. In 1965, S. Golomb asked a question (later popularized by Martin Gardiner) asking the following: how many monominoes need to be removed from the $8\times 8$ board in order to exclude any X-pentamino? The latter is a figure consisting of 5 cells shaped as a Greek cross. The answer is $10$. The example is given on the Figure \ref{fig:test2} (for more on this, see \cite{golomb}, Chapter 3). This leads to an analogous question for a general $n\times n$ grid graph, giving rise to the sequence $1, 2, 3, 4, 7, 10, 12, 16, 20, 24, 29, 35, 40, 47,\ldots$ (the offset is $n=3$; sequence A104519). The term $a_{n+2}=\gamma_{n}$ can also be interpreted as the domination number (size of minimum dominating set) in the $n\times n$ grid graph. Numerical calculations for $n$ up to $29$ are given in \cite{alanko}, where values for $n>19$ are due to the authors. Finally, Gon\c{c}alves at al. \cite{goncalves} proved Chang's conjecture, which (in particular) claimed that for $n\geq 16$,
\begin{eqnarray*}
\gamma_{n}=\bigg{\lfloor}\frac{(n+2)^2}{5}\bigg{\rfloor}-4.
\end{eqnarray*}

A closely related sequence, given by $n^2-\gamma_{n}$, was investigated in \cite{khovanova}. It is defined as a largest possible number in a saturated domino covering of a $n\times n$ board. A covering in dominoes is called \emph{saturated}, if removing any one of them leaves at least one cell uncovered (sequence A193764).\\

In general, there are quite a few integer sequences $a_{n}$ which have the following two characterisations:
\begin{itemize}
\item[1.]{\textit{Definition}.} They count the minimum (maximum) number of marked cells in the $n\times n$ grid graph (removed, distinguished, centers of trominoes, X-pentamineos etc.), such that the modified graph has the needed pre-defined property; 
\item[2.]\textit{{Formula}.} The proved (or conjectured) formula for this sequence obeys the asymptotics $a_{n}\sim r\cdot n^{2}$ for a certain $r\in\mathbb{Q}$.
\end{itemize}
Such sequences include A321684 (which differs from
$A104519(n+2) $by just two members, see \cite{crevals}),
A239072, A239231, A085577, A226918, A238381,
A288425, A152125, A322125, and more.

Further examples are A331968, A357357 and A357234
(maximal induced paths and cycles in square grids),
A369692 and A302488 (connected and total domination),
and A280984 (minimum size of a maximal domino packing).
Related square-board problems involving king or knight
adjacency give rise to A075561, A302401, A370428,
A000982, A357501, A357500 and A165143.
Another interesting example is A250000, concerning
peaceable armies of queens, for which the conjectured
quadratic coefficient is $7/48$ per army. The last example 
was the winner in the competition which sequence is to be assigned number 250.000 \cite{sloane}.
\begin{figure}[!htbp]
	\centering
	
	\begin{minipage}{0.46\textwidth}
		\centering
		\begin{tikzpicture}[x=7mm,y=7mm]
			
			\draw[step=1,black!75,line width=.35pt] (0,0) grid (7,7);
			\draw[black!85,line width=.5pt] (0,0) rectangle (7,7);
			
			\foreach \r in {1,...,7}{
				\foreach \c in {1,...,6}{
					\draw[black!35,line width=.45pt]
					(\c-0.5,{7.5-\r}) -- (\c+0.5,{7.5-\r});
				}
			}
			\foreach \c in {1,...,7}{
				\foreach \r in {1,...,6}{
					\draw[black!35,line width=.45pt]
					(\c-0.5,{7.5-\r}) -- (\c-0.5,{6.5-\r});
				}
			}
			
			\foreach \r in {1,...,7}{
				\foreach \c in {1,...,7}{
					\filldraw[fill=gray!55,draw=black!90,line width=.35pt]
					(\c-0.5,{7.5-\r}) circle (0.09);
				}
			}
			
		\end{tikzpicture}
			\captionof{figure}{Grid graph for $n=7$}
			\label{bas}
	\end{minipage}
	\hfill
	\begin{minipage}{0.46\textwidth}
		\centering
		\begin{tikzpicture}[x=7mm,y=7mm]
			
			\foreach \r/\c in {
				2/3,2/6,
				3/3,3/6,
				5/2,5/4,5/5,5/7,
				7/3,7/6
			}{
				\fill[gray!100]
				({\c-1},{8-\r})
				rectangle
				({\c},{9-\r});
			}
			
			\draw[step=1,black!75,line width=.35pt] (0,0) grid (8,8);
			\draw[black!85,line width=.5pt] (0,0) rectangle (8,8);
			
		\end{tikzpicture}
			\caption{No X-pentamino fits}
				\label{fig:test2}
	\end{minipage}
\end{figure}

\section{The proof}
\label{sec-pr}
Consider the grid graph  $\Gamma_{n}=(V_{n}, E_{n})$ (see Figure \ref{bas}). It has $n^2$ vertices and $2n(n-1)$ edges. Its vertices split into three sets, according to their degree. The latter can be equal to $4$, $3$, or $2$. Let us call these sets, respectively,  $\mathcal{A}$, $\mathcal{B}$, and $\mathcal{C}$. One has $|\mathcal{A}|=(n-2)^2$, $|\mathcal{B}|=4(n-2)$, $|\mathcal{C}|=4$.
Recall that a \emph{tree} is a connected graph without cycles. For a tree $\gamma=(v,e)$, one has $|e|=|v|-1$. A \emph{forest} is an acyclic graph, which then necessarily is a disjoint union of trees. If  $\gamma=(v,e)$ is a forest with $t$ trees, one has $|e|=|v|-t$.\\

Suppose now that the graph $\Gamma=(V,E)$ is obtained from the grid graph $\Gamma_{n}$ by removing the set $\mathcal{X}$ consisting of $s$ vertices with all incident edges. Assume $\Gamma$ is acyclic and consists of $T$ trees. Then
\begin{eqnarray}
|E|=|V|-T.
\label{acy}
\end{eqnarray}
Put $b=|\mathcal{X}\cap\mathcal{B}|$ and $c=|\mathcal{X}\cap\mathcal{C}|$. Then  $|\mathcal{X}\cap\mathcal{A}|=s-b-c$. Further, suppose the total number of pairs of adjacent cells in $\mathcal{X}$ is $u$ (three consecutive cells count as two pairs, four as twwo, and so on). By a direct calculation, $|V|=n^2-s$, $|E|=2n^2-2n-4s+b+2c+u$. According to (\ref{acy}), one has $2n^2-2n-4s+b+2c+u=n^2-s-T$.
And so,
\begin{eqnarray}
s=\frac{(n-1)^2+b+2c+u+T-1}{3}.
\label{form}
\end{eqnarray}
Now, one cell in $\mathcal{X}$ must lie on the outer border, otherwise there existed a cycle going round the whole board. This implies $b+c\geq 1$. Also, trivially, $T\geq 1$, $u\geq 0$. Therefore,
\begin{eqnarray*}
s\geq \bigg{\lceil}\frac{(n-1)^2+1}{3}\bigg{\rceil}.
\end{eqnarray*}

\section{Cases $n-1=6k\pm1$}

Suppose now $n=6k$ or $n=6k+2$, with $k\geq1$. Assume that equality
holds in the lower bound, so $s=L_n$. Since $(n-1)^2\equiv1\pmod3$,
we must have $3s=(n-1)^2+2$. Comparison with (\ref{form}) then gives
\begin{equation}
	b+2c+u+T-1=2.
	\label{form-n}
\end{equation}
We will see that this leads to a contradiction. Since $b+c\geq1$,
$T\geq1$ and $u\geq0$, there are three possibilities to explore.

\medskip
\noindent\textbf{i)} $(c,b,u,T)=(1,0,0,1)$. There are, accordingly,
no other grey cells on the border. Therefore, one can find a whole side
of the board which is white, and so are the two neighbouring boundary
cells beyond its ends (see Figure~\ref{iro-2}, left; all these cells are
coloured in electric blue). We say that such a side is \emph{in
	configuration} $\mathcal F$.

Suppose this is the left side. The cells $\mathbf A$ and $\mathbf B$
are necessarily grey. Every two consecutive cells in the second column
must contain a grey cell, otherwise there would be a white $2\times2$
cycle. If no two consecutive cells were grey, the second column would
therefore alternate between white and grey. Its two end cells are white,
which is impossible since $n$ is even. Thus $u\geq1$ -- contradiction.

\medskip
\noindent\textbf{ii)} $(c,b)=(0,1)$. A single non-corner boundary
deletion prevents at most two sides from being in configuration
$\mathcal F$. Hence at least two sides are in this configuration.
The previous argument gives a pair of consecutive grey cells next to
each of them. These pairs are distinct: for opposite sides they lie in
different rows or columns, and for neighbouring sides they have different
directions. Thus $u\geq2$, whereas (\ref{form-n}) gives $u+T=2$.
This is a contradiction again.

\medskip
\noindent\textbf{iii)} The difficult equality case is
$(c,b,u,T)=(0,2,0,1)$. Thus exactly two deleted cells lie on the boundary,
neither is a corner, no two deleted cells share a side, and the remaining
graph $\Gamma$ is a tree.

\begin{wrapfigure}{l}{0.58\textwidth}
	\includegraphics[width=0.58\textwidth]{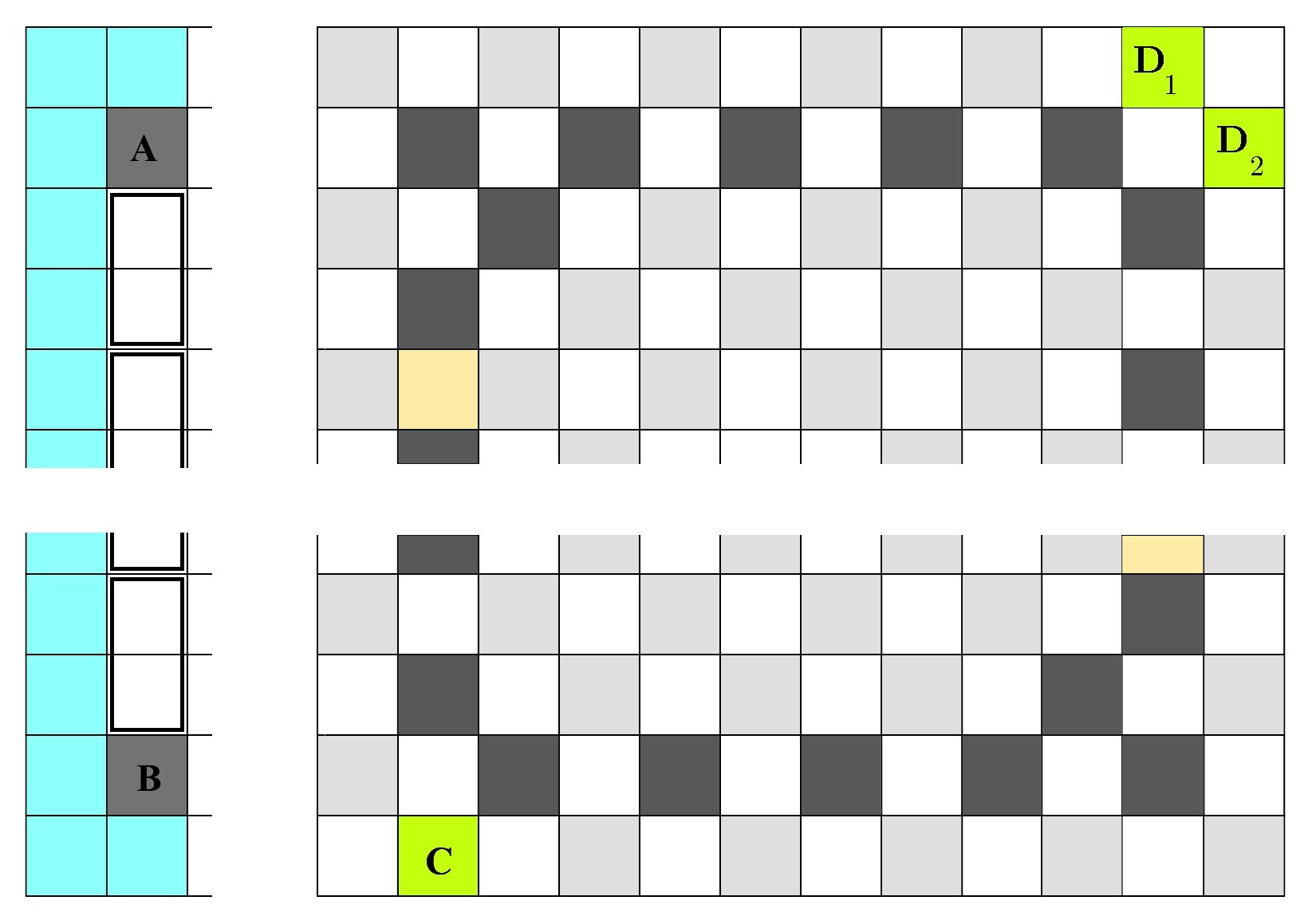}
	\caption{Supplementary figure for cases $n=6k$ or $n=6k+2$}
	\label{iro-2}
\end{wrapfigure}

\noindent\textbf{Step 1: locating the two boundary cells.}
No side can be in configuration $\mathcal F$. A non-corner boundary
deletion prevents its own side from being in this configuration, and can
prevent one further side only when it is immediately adjacent to a
corner. Since two deletions must account for all four sides, both are
immediately adjacent to corners, and those corners are opposite.

Use coordinates $(i,j)$ with $i$ counted from left to right and $j$ from
bottom to top. After applying a symmetry of the square, the two boundary
deletions are $\mathbf C=(1,2)$ and either
$\mathbf D=\mathbf D_1=(n,n-1)$ or
$\mathbf D=\mathbf D_2=(n-1,n)$. These are the two possibilities illustrated
in Figure~\ref{iro-2}, up to symmetry. In particular, $\mathbf C$ and
$\mathbf D$ have the same ordinary checkerboard colour. Let it be
light-grey.

\medskip
\noindent\textbf{Step 2: the diagonal-contact graph.}
Every $2\times2$ block contains a deleted cell. Because $u=0$, a block
contains at most two deleted cells, and if it contains two, they are
diagonally opposite. Let $r$ be the number of blocks containing two
deleted cells. The boundary deletions $\mathbf C,\mathbf D$ belong to two
blocks each, whereas every other deletion belongs to four. Counting
incidences gives $4s-4=(n-1)^2+r$. Since $3s=(n-1)^2+2$, we obtain
$r=s-2$.

Define a graph $K$ on vertex set $\mathcal X$ by joining deleted cells
which are diagonally opposite in a $2\times2$ block. Draw these edges as
segments between cell centres. They do not cross: crossing diagonals
would require four grey cells in a single block, contrary to $u=0$.

The graph $K$ is acyclic. Indeed, a simple cycle of $K$ would enclose a
white cell. To see this, take a leftmost vertex $(a,b)$ of the cycle.
Its two cycle neighbours are $(a+1,b-1)$ and $(a+1,b+1)$, so the cell
$(a+1,b)$ lies inside the cycle. It is white, since it shares a side with
$(a,b)$. The corner cells of the board are white and lie outside the
cycle. No edge of the surviving grid can cross a unit diagonal of $K$:
the interior of such a diagonal lies strictly inside an elementary grid
face, and its endpoints are deleted. Thus the white cell inside cannot
be connected to a corner outside, contrary to $T=1$.

Now $|V(K)|=s$ and $|E(K)|=r=s-2$, so $K$ consists of exactly two
connected components. If $\mathbf C$ and $\mathbf D$ were connected in
$K$, a simple diagonal path between them would lie in the interior of
the board except at its endpoints. It separates the two boundary arcs
between $\mathbf C$ and $\mathbf D$, each of which contains a white corner
cell. A path in $\Gamma$ between these arcs would have to cross the grey
diagonal path, which is again impossible. Consequently the two components
of $K$, denoted by $K_C,K_D$, contain $\mathbf C$ and $\mathbf D$,
respectively.

Every diagonal edge preserves checkerboard colour. Each component of
$K$ contains one of $\mathbf C,\mathbf D$, and these have the same colour.
Therefore $\mathcal X\subseteq Q$, where $Q$ is the full checkerboard
class of cells $(i,j)$ with $i+j$ odd. Let $H$ be the graph on $Q$ joining
diagonally opposite cells of each $2\times2$ block. Thus
$K=H[\mathcal X]$. Put $Y=Q\setminus\mathcal X$. Since every block contains
a deletion, its two $Q$-cells cannot both belong to $Y$. Hence $Y$ is an
independent set in $H$.

There is now a set of cells which are forced to be grey. Let $F$ consist
of the cells of $Q$ on the inner frame, namely the boundary of
$[2,n-1]^2$. Every cell of $F$ has two diagonal neighbours on an adjacent
outer side, at least one of which is white. In the corresponding
$2\times2$ block the other two cells lie outside $Q$ and are also white.
Thus the cell of $F$ must be grey.

The cells $(3,n-2)$ and $(n-2,3)$ are forced as well. The forced cell
$(2,n-1)$ has only one possible grey diagonal neighbour, $(3,n-2)$; its
other three diagonal neighbours are boundary cells different from
$\mathbf C,\mathbf D$. It cannot be isolated, since the two components of
$K$ already contain $\mathbf C$ and $\mathbf D$. Hence $(3,n-2)$ is grey.
The same argument at $(n-1,2)$ gives $(n-2,3)$.

For $n=6$, these forced cells already give a contradiction: the four grey
cells $(2,3),(3,4),(4,3),(3,2)$ form a cycle of $K$. Hence from now on we
may assume $n\geq8$.

Each of $\mathbf C,\mathbf D$ is a leaf of its component. Indeed,
$\mathbf C$ has only one grey diagonal neighbour, $(2,3)$, and the same
statement holds at $\mathbf D$. Hence $K_C-\mathbf C$ and
$K_D-\mathbf D$ are disjoint connected graphs contained in the closed
inner square. Label every cell of $F$ by $C$ or $D$ according as it
belongs to $K_C$ or $K_D$, and read the labels in cyclic order.

\begin{lem}\label{lem:twotransitions}
	The cyclic sequence of labels on $F$ changes exactly twice. One change
	occurs on the left or top side of the inner frame, and the other on the
	bottom or right side.
\end{lem}
\begin{proof}
	The labels cannot interlace. Otherwise four cells would occur in cyclic
	order with labels $C,D,C,D$. A path joining the two $C$-cells in
	$K_C-\mathbf C$ and a path joining the two $D$-cells in
	$K_D-\mathbf D$ would have alternating endpoints on the boundary of a
	square and would have to intersect. This contradicts the fact that the
	two components are disjoint. Both labels occur, so each occupies one
	consecutive part of the frame. There are exactly two changes.
	
	The cells $(2,3),(3,2)$ are diagonally joined and both have label $C$.
	Similarly, $(n-2,n-1),(n-1,n-2)$ both have label $D$. At the other two
	corners, the additional forced cells $(3,n-2),(n-2,3)$ join the frame
	cells around each corner, so the label does not change there either.
	Each of the two frame arcs from the $C$-corner to the $D$-corner must
	contain a change. One arc runs along the left and top sides, and the
	other along the bottom and right sides.
\end{proof}

A change on a straight side has a useful local interpretation. For
example, suppose it occurs on the left side between $(2,j-1)\in K_C$
and $(2,j+1)\in K_D$. Then both $(1,j)$ and $(3,j)$ belong to $Y$:
if either were grey, it would join the two components. We call $(3,j)$
the \emph{separator} corresponding to this change. The frame cell
$(2,j)$ between them lies outside $Q$ and is white. The analogous
statement holds on the other sides. Thus there are exactly two
transitions, with their corresponding inward white separator cells.
The light-yellow cells in Figure~\ref{iro-2} indicate the positions of
the transitions on the frame.

As a local check, exactly one of $(3,4),(4,3)$ is grey. At least one is
grey because their $2\times2$ block must contain a deletion, while both
cannot be grey because together with $(2,3),(3,2)$ they would form a
cycle of $K$. Likewise, exactly one of $(n-2,n-3),(n-3,n-2)$ is grey.

\medskip
We now finish the contradiction using only parity and the handshaking
lemma. For an interior cell $z\notin Q$, its four side-neighbours belong
to $Q$ and form a $4$-cycle in $H$.

\begin{lem}\label{lem:localpattern}
	If these four neighbours contain grey cells from both $K_C$ and $K_D$,
	then they occur cyclically in the form $C,Y,D,Y$, up to rotation and
	interchange of $C,D$. In particular, the two white neighbours are
	opposite.
\end{lem}
\begin{proof}
	Grey vertices from different components cannot be adjacent in the
	$4$-cycle, so they must be opposite. Either remaining vertex, if grey,
	would be adjacent to both and would join their components. Thus both
	remaining vertices belong to $Y$.
\end{proof}

Define a graph $S$ on $Y$ as follows. For every interior cell $z\notin Q$
whose four $Q$-neighbours contain both component labels, join the two
opposite white neighbours given by Lemma~\ref{lem:localpattern}.
Geometrically this edge is the straight segment through $z$.

\begin{lem}
	Every interior vertex of $S$ has even degree. A boundary vertex has
	degree $0$ or $1$, and degree $1$ occurs precisely at a transition on
	the inner frame. Consequently $S$ has exactly two odd-degree vertices.
\end{lem}
\begin{proof}
	Let $y=(i,j)\in Y$ be an interior cell. Since all $Q$-cells on the inner
	frame are grey, $3\leq i,j\leq n-2$. Since $Y$ is independent in $H$,
	all four diagonal neighbours of $y$ are grey. Read their component
	labels cyclically around $y$. An edge of $S$ is incident with $y$
	exactly when two consecutive labels are different. The cell through
	which this edge passes is interior, and the local pattern lemma applies.
	Thus $\deg_S(y)$ is the number of changes in a cyclic binary sequence,
	and is even.
	
	If $y\in Y$ is a non-corner boundary cell, only one edge of $S$ can be
	incident with it: the edge directed into the board. It is present exactly
	when the two inward diagonal neighbours of $y$ have different component
	labels. Except for the white cell across a corner from $\mathbf C$ or
	$\mathbf D$, these neighbours are consecutive cells of $F$. At the
	exceptional positions they belong to the same component: for instance,
	the neighbours of $(2,1)$ are $\mathbf C$ and $(3,2)$, both in $K_C$.
	The analogous statement holds at $\mathbf D$. The two corners in $Q$
	have degree zero, since no segment through an interior cell can have a
	corner as its endpoint. Lemma~\ref{lem:twotransitions} now gives exactly
	two boundary vertices of degree $1$, and no other odd-degree vertices.
\end{proof}

By the handshaking lemma applied to each connected component of $S$,
its two odd-degree vertices must belong to the same component. Finally,
split $Q$ into $Q_A$, consisting of cells with $i$ odd and $j$ even, and
$Q_B$, consisting of cells with $i$ even and $j$ odd. Every edge of $S$
joins cells differing by $(\pm2,0)$ or $(0,\pm2)$. Hence each connected
component of $S$ is contained entirely in $Q_A$ or entirely in $Q_B$.

But the odd-degree vertex corresponding to the left/top transition lies
on the left boundary $i=1$ or the top boundary $j=n$; since $n$ is even,
it belongs to $Q_A$. The other lies on the bottom boundary $j=1$ or the
right boundary $i=n$, and belongs to $Q_B$. They therefore cannot belong
to the same connected component of $S$. This is the required
contradiction.

\section{Cases $n-4=6k\pm 1$}
This section needs almost no words. In sufficies to present an example for $n=9$, an examples for $n=11$ (Fig. \ref{ex2}) and $15$ (Fig. \ref{fig:global}, left, dashed contour). The last two have identical pattern in bottom-and-left $2\times 2$ strip.  Thus, the induction  $n\rightarrow n+6$ using a local surgery presented in Fig. \ref{fig:global} works in both cases.
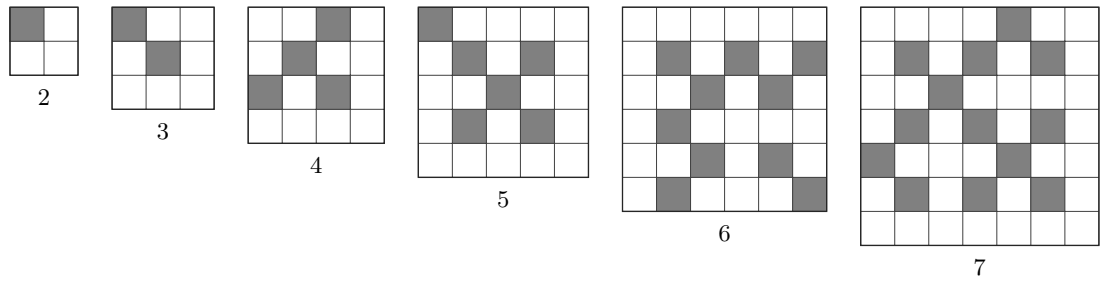
\begin{figure}[!htbp]
	\centering
	
	\begin{tikzpicture}[x=4.5mm,y=4.5mm]
		
		
		\newcommand{\DrawSmallBoard}[3]{%
			\begin{scope}[shift={(#2,{7-#1})}]
				\fill[white] (0,0) rectangle (#1,#1);
				
				\foreach \r/\c in {#3}{
					\fill[gray!100]
					({\c-1},{#1-\r})
					rectangle
					({\c},{#1-\r+1});
				}
				
				\draw[step=1,black!75,line width=.35pt]
				(0,0) grid (#1,#1);
				
				\draw[black!85,line width=.5pt]
				(0,0) rectangle (#1,#1);
				
				\node[font=\small] at ({#1/2},-.65) {$#1$};
			\end{scope}%
		}
		
		\DrawSmallBoard{2}{0}{
			1/1
		}
		
		\DrawSmallBoard{3}{3}{
			1/1,
			2/2
		}
		
		\DrawSmallBoard{4}{7}{
			1/3,
			2/2,
			3/1,3/3
		}
		
		\DrawSmallBoard{5}{12}{
			1/1,
			2/2,2/4,
			3/3,
			4/2,4/4
		}
		
		\DrawSmallBoard{6}{18}{
			2/2,2/4,2/6,
			3/3,3/5,
			4/2,
			5/3,5/5,
			6/2,6/6
		}
		
		\DrawSmallBoard{7}{25}{
			1/5,
			2/2,2/4,2/6,
			3/3,
			4/2,4/4,4/6,
			5/1,5/5,
			6/2,6/4,6/6
		}
		
	\end{tikzpicture}
	
	\caption{The needed configuration of grey squares
		for $2\leq n\leq7$.}
	\label{ex1}
	
\end{figure}

\newcommand{\DrawSquareExample}[3]{%
	\begin{tikzpicture}[x=.42cm,y=.42cm,baseline=(current bounding box.south)]
		\foreach \r/\c in {#2}{
			\fill[gray!100] (\c-1,{#1-\r}) rectangle (\c,{#1-\r+1});
		}
		\draw[step=1,black!80,line width=.5pt] (0,0) grid (#1,#1);
		\node at ({#1/2},-0.9) {$#3$};
	\end{tikzpicture}%
}

\begin{figure}[ht]
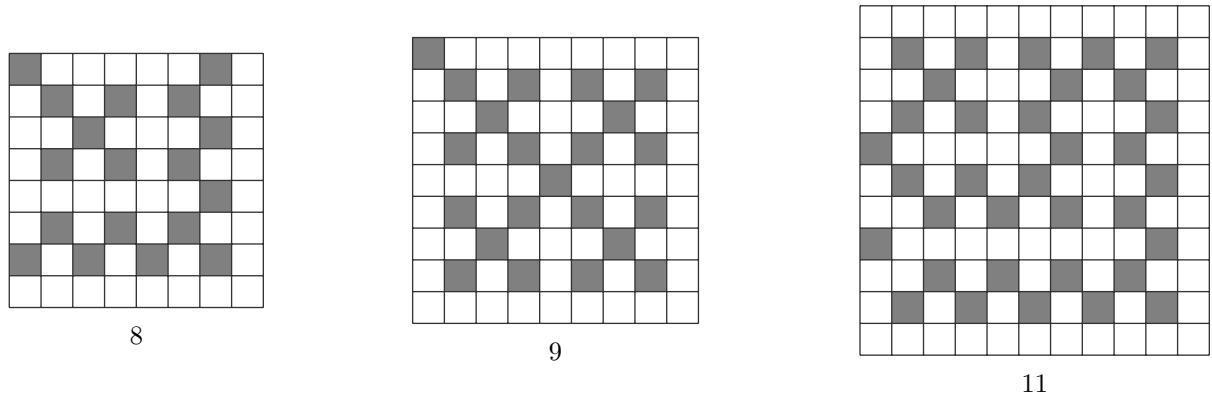

	\centering
	
	\begin{minipage}{0.28\textwidth}
		\centering
		\DrawSquareExample{8}{
			1/1,1/7,
			2/2,2/4,2/6,
			3/3,3/7,
			4/2,4/4,4/6,
			5/7,
			6/2,6/4,6/6,
			7/1,7/3,7/5,7/7
		}{8}
	\end{minipage}
	\hfill
	\begin{minipage}{0.31\textwidth}
		\centering
		\DrawSquareExample{9}{
			1/1,
			2/2,2/4,2/6,2/8,
			3/3,3/7,
			4/2,4/4,4/6,4/8,
			5/5,
			6/2,6/4,6/6,6/8,
			7/3,7/7,
			8/2,8/4,8/6,8/8
		}{9}
	\end{minipage}
	\hfill
	\begin{minipage}{0.37\textwidth}
		\centering
		\DrawSquareExample{11}{
			2/2,2/4,2/6,2/8,2/10,
			3/3,3/7,3/9,
			4/2,4/4,4/6,4/10,
			5/1,5/7,5/9,
			6/2,6/4,6/6,6/10,
			7/3,7/5,7/7,7/9,
			8/1,8/10,
			9/3,9/5,9/7,9/9,
			10/2,10/4,10/6,10/8,10/10
		}{11}
	\end{minipage}
		\caption{The needed configuration of grey squares for $n=8,9,11$.}
		\label{ex2}
\end{figure}

\begin{figure}[!htbp]
	\centering
	\begin{minipage}[c]{.49\linewidth}\centering
		\begin{tikzpicture}[x=3.85mm,y=3.85mm,baseline=(current bounding box.center)]
			\fill[white] (0,0) rectangle (21,21);
			\foreach \r/\c in {1/9,2/2,2/4,2/6,2/8,2/10,2/12,2/14,3/3,3/7,3/11,3/13,4/2,4/4,4/6,4/8,4/10,4/14,5/5,5/11,5/13,6/2,6/4,6/6,6/8,6/10,6/14,7/3,7/7,7/11,7/13,8/2,8/4,8/6,8/8,8/10,8/14,9/11,9/13,10/2,10/4,10/6,10/8,10/10,10/14,11/3,11/5,11/7,11/9,11/11,11/13,12/1,12/14,13/3,13/5,13/7,13/9,13/11,13/13,14/2,14/4,14/6,14/8,14/10,14/12,14/14}{\fill[gray!100] (\c-1,21-\r) rectangle (\c,21-\r+1);}
			\foreach \r/\c in {2/16,2/18,2/20,3/17,3/19,4/16,4/20,5/17,5/19,6/16,6/20,7/17,7/19,8/16,8/20,9/17,9/19,10/16,10/20,11/17,11/19,12/16,12/20,13/17,13/19,14/16,14/20,15/17,15/19,16/2,16/4,16/6,16/8,16/10,16/12,16/14,16/16,16/20,17/3,17/5,17/7,17/9,17/11,17/13,17/15,17/17,17/19,18/1,18/20,19/3,19/5,19/7,19/9,19/11,19/13,19/15,19/17,19/19,20/2,20/4,20/6,20/8,20/10,20/12,20/14,20/16,20/18,20/20}{\fill[blue!32] (\c-1,21-\r) rectangle (\c,21-\r+1);}
			\foreach \i in {0,...,21}{\draw[black!60,line width=.22pt] (\i,0)--(\i,21);\draw[black!60,line width=.22pt] (0,\i)--(21,\i);}
			\draw[black!70,line width=.5pt] (0,0) rectangle (21,21);
			\draw[densely dashed,black!78,line width=1.4pt] (0,6) rectangle (15,21);
			\draw[orange!95!black,line width=1.0pt] (0.5,9.5) circle[radius=.38];
			\draw[orange!95!black,line width=1.0pt] (3.5,7.5) circle[radius=.38];
			\draw[orange!95!black,line width=1.0pt] (3.5,5.5) circle[radius=.38];
		\end{tikzpicture}
		
		\small After an addition, before surgery
	\end{minipage}\hfill
	\begin{minipage}[c]{.49\linewidth}\centering
		\begin{tikzpicture}[x=3.85mm,y=3.85mm,baseline=(current bounding box.center)]
			\fill[white] (0,0) rectangle (21,21);
			\foreach \r/\c in {1/9,2/2,2/4,2/6,2/8,2/10,2/12,2/14,3/3,3/7,3/11,3/13,4/2,4/4,4/6,4/8,4/10,4/14,5/5,5/11,5/13,6/2,6/4,6/6,6/8,6/10,6/14,7/3,7/7,7/11,7/13,8/2,8/4,8/6,8/8,8/10,8/14,9/11,9/13,10/2,10/4,10/6,10/8,10/10,10/14,11/3,11/5,11/7,11/9,11/11,11/13,12/2,12/14,13/3,13/5,13/7,13/9,13/11,13/13,14/2,14/6,14/8,14/10,14/12,14/14,15/3,15/5}{\fill[gray!100] (\c-1,21-\r) rectangle (\c,21-\r+1);}
			\foreach \r/\c in {2/16,2/18,2/20,3/17,3/19,4/16,4/20,5/17,5/19,6/16,6/20,7/17,7/19,8/16,8/20,9/17,9/19,10/16,10/20,11/17,11/19,12/16,12/20,13/17,13/19,14/16,14/20,15/17,15/19,16/2,16/6,16/8,16/10,16/12,16/14,16/16,16/20,17/3,17/5,17/7,17/9,17/11,17/13,17/15,17/17,17/19,18/1,18/20,19/3,19/5,19/7,19/9,19/11,19/13,19/15,19/17,19/19,20/2,20/4,20/6,20/8,20/10,20/12,20/14,20/16,20/18,20/20}{\fill[blue!32] (\c-1,21-\r) rectangle (\c,21-\r+1);}
			\foreach \i in {0,...,21}{\draw[black!60,line width=.22pt] (\i,0)--(\i,21);\draw[black!60,line width=.22pt] (0,\i)--(21,\i);}
			\draw[black!70,line width=.5pt] (0,0) rectangle (21,21);
			\draw[teal!85!black,line width=1.0pt] (1.5,9.5) circle[radius=.38];
			\draw[teal!85!black,line width=1.0pt] (2.5,6.5) circle[radius=.38];
			\draw[teal!85!black,line width=1.0pt] (4.5,6.5) circle[radius=.38];
		\end{tikzpicture}
		
		\small After surgery.
	\end{minipage}
	\caption{The full-board operation $15\longrightarrow21$. Orange circles mark the three deleted cells to be restored, and green the three new deletions.}
	\label{fig:global}
\end{figure}
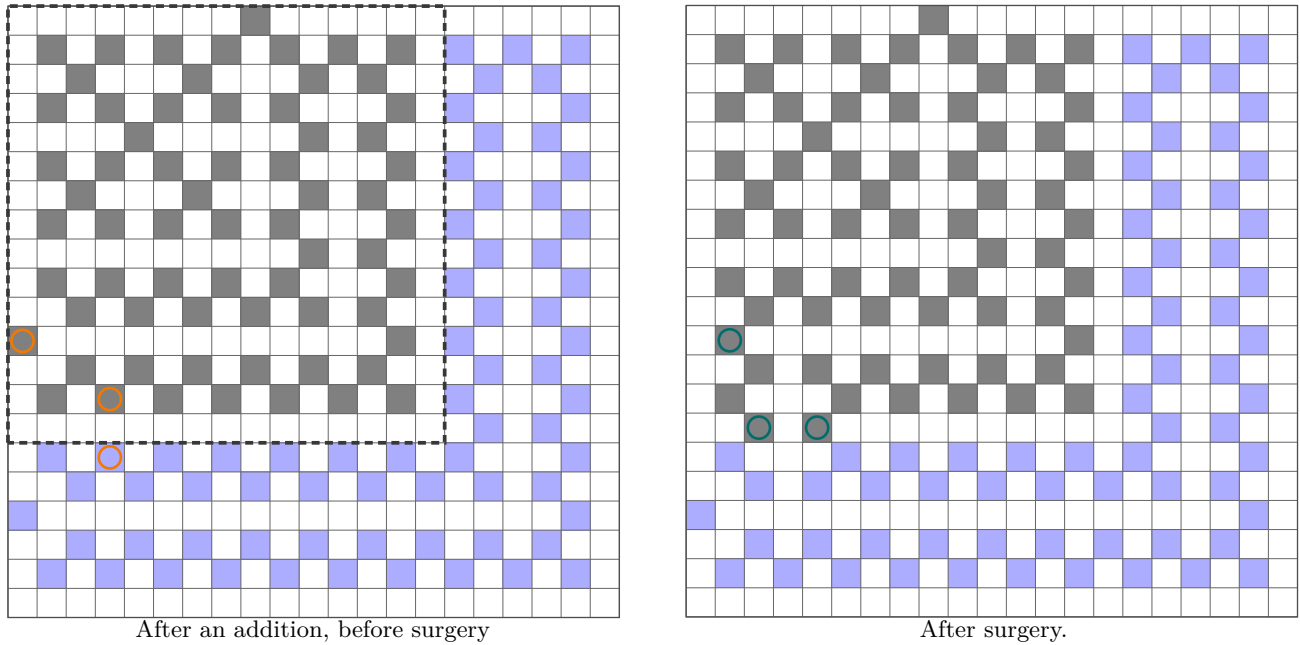


\begin{thebibliography}{99}

\bibitem{alanko}{\sc S. Alanko, S. Crevals, A. Isopoussu, P.R.J. \"Osterg\aa rd, V. Pettersson}, Computing the Domination Number of Grid Graphs,, {\it The Electronic Journal of Combinatorics} 18 (2011), \# P141.

\bibitem{crevals}{\sc S. Crevals, P.R.J. \"Osterg\aa rd}, Independent domination of grids. {\it Discrete Math.} {\bf 338} (8) (2015), 1379--1384.

\bibitem{golomb}{\sc S.W. Golomb}, {\it Polyominoes: Puzzles, Patterns, Problems, and Packings}, Revised and Expanded Second Edition Cambridge University Press, 1998.

\bibitem{goncalves}{\sc D. Gon\c alves, A. Pinlou,  M. Rao, S. Thomass\'e}, The domination number of grids. {\it SIAM J. Discrete Math.} {\bf 25} (3) (2011), 1443--1453.


\bibitem{khovanova}{\sc A. Buchanan, T. Khovanova, A. Ryba},
Saturated domino coverings, {\it Geombinatorics} {\bf 24} (2) (2014), 52–-67.

\bibitem{LienFuShih}
{\sc M.-Y. Lien, H.-L. Fu and C.-H. Shih},
\emph{The decycling number of \(P_m\square P_n\)},
Discrete Mathematics, Algorithms and Applications \textbf{6} (2014), no.~3, Article 1450033.

\bibitem{Luccio}{\sc F.L. Luccio}, Almost exact minimum feedback vertex sets in meshes and butter-
flies, Information Processing Letters 66 (1998) 59–64.


\bibitem{MadelaineStewart}
{\sc F.~R. Madelaine and I.~A. Stewart},
\emph{Improved upper and lower bounds on the feedback vertex numbers of grids and butterflies},
Discrete Mathematics \textbf{308} (2008), 4144--4164.

\bibitem{sloane}{\sc N.J.A. Sloane}, The On-Line Encyclopedia of Integer Sequences. {\it Notices Amer. Math. Soc.} {\bf 65}(9),  1062--1074 (2018).

\bibitem{oeis}{\sc OEIS}, The Online Encyclopedia of Integer Sequences, Sequences A010060, A000120, A018819, \url{http://www.research.att.com/~njas/sequences/}


\bibitem{weisstein}{\sc Weisstein, Eric W.}, Feedback Vertex Set Number, \url{https://mathworld.wolfram.com/FeedbackVertexSetNumber.html}
\end{thebibliography}
\end{document}